\documentclass[11pt,reqno]{amsart}

\usepackage{a4wide}
\usepackage{hyperref,amsmath,amsfonts,mathscinet,amssymb,amsthm}
\usepackage{enumitem}
\usepackage{constants}
\usepackage{xcolor}
\usepackage{tikz-cd}
\usepackage{todonotes}

\theoremstyle{plain}
\newtheorem{theorem}{Theorem}[section]

\newtheorem{lemma}[theorem]{Lemma}

\theoremstyle{definition}

\theoremstyle{remark}

\numberwithin{equation}{section}

\newcommand\ddfrac[2]{\frac{\displaystyle #1}{\displaystyle #2}}
\newcommand\esup{\operatornamewithlimits{ess\,sup}}
\newcommand\Id{\operatorname{I}}
\newcommand\Ces{\operatorname{Ces}}
\newcommand\Cop{\operatorname{Cop}}

\newcommand\supp{\operatorname{supp}}
\allowdisplaybreaks

\begin{document}

\title{Embeddings Between Weighted Ces\`{a}ro Function Spaces}

\author{Tu\u{g}\c{c}e \"{U}nver}

\email[T.~\"{U}nver]{tugceunver@kku.edu.tr}
\urladdr{0000-0003-0414-8400}

%
\address{Department of Mathematics,
Faculty of Science and Arts,
 Kirikkale University,
  71450 Yahsihan,
  Kirikkale, Turkey}



\begin{abstract}
In this paper, we give the characterization of the embeddings between weighted Ces\`{a}ro function spaces. The proof is based on the duality technique, which reduces this problem to the characterizations of some direct and reverse Hardy-type inequalities and iterated Hardy-type inequalities.
\end{abstract}

\date{\today}

\maketitle

\section{Introduction}

Our principle goal in this paper is to obtain two-sided estimates of the best constant $c$ in the inequality 
\begin{align}\label{emb. inequality}
&\bigg(\int_0^{\infty} \bigg(\int_0^t f(s)^{p_2} v_2(s)^{p_2} 
ds\bigg)^{\frac{q_2}{p_2}} u_2(t)^{q_2} dt\bigg)^{\frac{1}{q_2}} \nonumber 
\\
&\hspace{3cm}\leq c \bigg(\int_0^{\infty} \bigg(\int_0^t f(s)^{p_1} 
v_1(s)^{p_1} ds\bigg)^{\frac{q_1}{p_1}} u_1(t)^{q_1} 
dt\bigg)^{\frac{1}{q_1}},
\end{align}
where $0 < p_1, p_2, q_1, q_2 < \infty$ and $u_1, u_2, v_1, v_2$ are non-negative measurable functions. 

Let $X$ and $Y$ be quasi normed vector spaces. If $X \subset Y$ and the identity operator is continuous from $X$ to $Y$, that is, there exists a positive constant $c$ such that $\|\Id(z)\|_Y \leq c\|z\|_X$ for all $z \in X$, we say that $X$ is embedded into $Y$ and write $X \hookrightarrow Y$. We denote by $\mathcal M^+$, the set of all non-negative measurable functions on $(0, \infty)$. A weight is a function such that measurable, positive and finite a.e on $(0,\infty)$ and we will denote the set of weights by $\mathcal W$. 

We denote by $\Ces_{p,q}(u,v)$, the weighted Ces\`{a}ro function spaces and $\Cop_{p,q}(u,v)$, the weighted Copson function spaces, the collection of all functions on $\mathcal{M}^+$ such that
\begin{equation*}
\|f\|_{\Ces_{p,q}(u,v)} = \bigg(\int_0^{\infty} \bigg(\int_0^t f(s)^{p} v(s)^{p} 
ds\bigg)^{\frac{q}{p}} u(t)^{q} dt\bigg)^{\frac{1}{q}} < \infty,
\end{equation*}
and
\begin{equation*}
\|f\|_{\Cop_{p,q}(u,v)} = \bigg(\int_0^{\infty} \bigg(\int_t^{\infty} f(s)^{p} v(s)^{p} 
ds\bigg)^{\frac{q}{p}} u(t)^{q} dt\bigg)^{\frac{1}{q}} < \infty,
\end{equation*}
respectively, where $p,q \in (0, \infty)$, $u \in \mathcal{M}^+$ and $v \in \mathcal W$. Then with 
this denotation, we can formulate the main aim of this paper as the characterization of the 
embeddings between weighted Ces\`{a}ro function spaces, that is, 
\begin{equation*}
\Ces_{p_1, q_1}(u_1, v_1) \hookrightarrow \Ces_{p_2, q_2}(u_2, v_2).
\end{equation*} 

The classical Ces\`{a}ro function spaces $\Ces_{1,p}(x^{-1}, 1)$ have been 
defined by Shiue in \cite{shiue} and it was shown in \cite{hashus} that these 
spaces are Banach spaces when $p > 1$.

In \cite{bennett1996}, it was shown that $\Ces_{1, p}(x^{-1}, 1)$ and 
$\Cop_{1,p}(1, x^{-1})$ coincide when $1 < p < \infty$ and the dual of the 
$\Ces_{1, p}(x^{-1}, 1)$ function spaces is given with a simpler description 
than in \cite{syzhanglee} as a remark. 

During the past decade, these spaces have not been studied to a high degree but 
recently Astashkin and Maligranda began to examine the properties of classical 
Ces\`{a}ro and Copson spaces in various aspects (\cite{asmal12, asmal13, 
	astas5, astashkinmaligran11, astasmal2008, astasmal2009, astasmal2010, 
	astasmalig10}), for the detailed information see the survey paper 
\cite{asmalsurvey}. In \cite{astasmal2009}, they gave the proof of the 
characterization of dual spaces of classical Ces\`{a}ro function spaces. Later, 
in  \cite{kamkub} authors computed the dual norm of the spaces $Ces_{1, p}(w, 
1)$ generated by an arbitrary positive weight $w$, where $1 < p < \infty$. In \cite{bmm2018}, 
factorizations of spaces $\Ces_{1, p}(1, x^{-1},v)$ and $\Cop_{1,p}(x^{-1}, v)$ are presented. 

In their newly papers, Le\'{s}nik and Maligranda (\cite{LesMal1, LesMal2, 
	LesMal3, LesMal4}) started the study of these spaces in an abstract setting and 
they replaced the role of $L_p$ spaces with a more general function space $X$.

A Banach ideal space $X$ on $(0, \infty)$ is a Banach space contained in $\mathcal M^+$ which 
satisfies the monotonicity property, that is, $f, g \in \mathcal M^+$, $f \leq g$ a.e  on $(0, 
\infty)$ and $g \in X$ then $f \in X$ and $\|f\| \leq \|g\|$.

For a Banach ideal space $X$ on $(0, \infty)$, Le\'{s}nik and Maligranda 
defined an abstact Ces\`{a}ro space $CX$ as
\begin{equation*}
CX = \{f \in \mathcal M^+, Cf \in X\}
\end{equation*}  
with the norm $\|f\|_{CX} = \|Cf\|_X$ and an abstract Copson space $C^*X$ as 
\begin{equation*}
C^*X = \{f \in \mathcal M^+,  C^*f \in X\}
\end{equation*}  
with the norm $\|f\|_{C^*X} = \|C^*f\|_X$, where
\begin{equation*}
Cf(x) = \frac{1}{x} \int_0^x f(t) dt \quad \text{and} \quad C^*f(x) = 
\int_x^{\infty} \frac{f(t)}{t} dt, \quad x\in (0, \infty).
\end{equation*} 
Moreover, in \cite{curbera16} abstract Ces\`{a}ro spaces were considered for rearrangement 
invariant spaces. 

Note that taking $X= L_p$, the definition of abstract spaces is related to our definition in the 
following way: $CL_p = \Ces_{1, p}(x^{-1}, 1)$.

Let $X$ and $Y$ be (quasi-) Banach spaces of measurable functions on $(0,\infty)$. Denote by 
$\textsf{M}(X, Y)$, the space of all multipliers, that is,
\begin{equation*}
\textsf{M}(X,Y):= \{f: \, f\cdot g \in Y\quad  \text{for all} \quad g\in X\}.
\end{equation*}
The K\"{o}the dual $X'$ of $X$ is defined as the space $\textsf{M}(X, L_1)$ of multipliers into 
$L_1$. 

The space of all multipliers from $X$ into $Y$ is a quasi normed space with the quantity
\begin{equation*}
\|f\|_{\textsf{M}(X, Y)} := \sup_{g \neq 0} \frac{\|fg\|_Y}{\|g\|_X}.
\end{equation*}
Now, define a weighted space $Y_f = \{g: \quad f\cdot g \in Y, \quad f \in \mathcal{W}\}$. Then 
\begin{equation*}
\|f\|_{\textsf{M}(X,Y)} = \sup_{g \neq 0} \frac{\|g\|_{Y_f}}{\|g\|_X} = \|\Id\|_{X\rightarrow 
	Y_f}.
\end{equation*}

We should mention that, in \cite{grosse}, it is stated that the characterizing the 
multipliers between Ces\`{a}ro and Copson spaces are difficult and note that 
the weighted Ces\`{a}ro and Copson spaces are related to the spaces $C$ and $D$ 
defined in \cite{grosse} as follows:
\begin{equation*}
\Ces_{p, q}(u, v)= C(p, q, u)_v \quad \text{and} \quad \Cop_{p, q}(u, v) = 
D(p, q, u)_v.
\end{equation*}
Among all, recently in \cite{LesMal4}, multipliers between  $\Ces_{1, p}(x^{-1}, 1)$ and 
$\Cop_{1,q}(1, x^{-1})$ is given when $1 < q \leq p \leq \infty$. 

With this motivation in \cite{GogMusUn}, the embeddings between weighted Copson and 
Ces\`{a}ro function spaces, that is, 
\begin{equation*}
\Cop_{p_1, q_1}(u_1, v_1) \hookrightarrow \Ces_{p_2, q_2}(u_2, v_2)
\end{equation*}
have been characterized under the restriction $p_2 \leq q_2$ arises from duality approach. Also, 
using these results pointwise multipliers of weighted Ces\`{a}ro and Copson spaces is given in 
\cite{GogMusUn2}. We want to extend these results. In \cite{GogMusUn} Sawyer duality principle reduced the problem of embeddings to solutions of 
iterated Hardy-type inequalities of the forms
\begin{equation*}
\bigg( \int_0^{\infty} \bigg(\sup_{s \in (0,t)} u(s) \int_s^{\infty} f \bigg)^q w(t) 
dt\bigg)^{\frac{1}{q}} \leq C
\bigg(\int_0^{\infty} f(t)^p v(t) dt\bigg)^{\frac{1}{p}},
\end{equation*} 
and 
\begin{equation*}
\bigg( \int_0^{\infty} \bigg(\int_0^t \bigg( \int_s^{\infty} f\bigg)^m u(s) ds 
\bigg)^{\frac{q}{m}} w(t) dt\bigg)^{\frac{1}{q}} \leq C \bigg(\int_0^{\infty} f(t)^p v(t) 
dt\bigg)^{\frac{1}{p}}
\end{equation*}
where $0 < m, q \leq \infty$ and $1 < p < \infty$ (see, for instance \cite{gmp, 
	GogMusIHI, GogMusISI, gop}).  

Using the same approach as in \cite{GogMusUn}, inequality \eqref{emb. inequality} reduces 
to the characterization of iterated inequalities (containing iterated Copson-type operators) of the 
following type, 
\begin{equation*}
\bigg( \int_0^{\infty} \bigg(\sup_{s \in (t, \infty)} u(s) \int_s^{\infty} f \bigg)^q w(t) 
dt\bigg)^{\frac{1}{q}} \leq C
\bigg(\int_0^{\infty} f(t)^p v(t) dt\bigg)^{\frac{1}{p}},
\end{equation*} 
and 
\begin{equation*}
\bigg( \int_0^{\infty} \bigg(\int_t^{\infty} \bigg( \int_s^{\infty} f\bigg)^m u(s) ds 
\bigg)^{\frac{q}{m}} w(t) dt\bigg)^{\frac{1}{q}} \leq C \bigg(\int_0^{\infty} f(t)^p v(t) 
dt\bigg)^{\frac{1}{p}},
\end{equation*}
where $1 < p < \infty$ and $0 < q, m < \infty $. Until recently the solutions of these 
problems were not known but not long ago different characterizations have been given for 
these inequalities, see \cite{Krepela, GogMusIHI,GogMusISI,Mus,KrePick}, therefore now we 
are able to continue this study. We will use characterizations from \cite{Krepela} and 
\cite{KrePick}.  

In order to shorten the formulas and simplify the notation, we will characterize the following 
inequality:
\begin{equation}\label{main inequality}
\bigg(\int_0^{\infty} \bigg(\int_0^t f(s)^p v(s) ds\bigg)^{\frac{q}{p}} u(t) dt 
\bigg)^{\frac{1}{q}}  \leq C \bigg( \int_0^{\infty} \bigg(\int_0^t f(s)  ds \bigg)^{\theta} 
w(t) dt \bigg)^{\frac{1}{\theta}}.
\end{equation}
It is easy to see that taking parameters $p = \frac{p_2}{p_1}$, \, $q = \frac{q_2}{p_1}$, 
\, $\theta = \frac{q_1}{p_1}$ and weights $v = v_1^{-p_2} v_2^{p_2}$, \, $u = 
u_2^{q_2}$,\,  $w = u_1^{q_1}$, we can obtain the characterization of inequality 
\eqref{emb. inequality}.

When $p= q$ or $\theta = 1$, \eqref{main inequality} has been characterized in 
\cite{GogMusUn} by using direct and reverse Hardy-type inequalities. Unfortunately in this paper we 
will solve this problem under the restriction $p < q$ arising from the techniques we 
used, we will deal with the case when $q < p$ in the future paper with a different approach. On the 
other hand we always assume that $p<1$, since otherwise inequality \eqref{main inequality} 
holds only for trivial functions (see Lemma~\ref{triviality}).  

We adopt the following usual conventions. Throughout the paper we put $0/0 = 0$, $0 \cdot (\pm 
\infty) = 0$ and $1 / (\pm\infty) =0$. For $p\in (1,\infty)$, we define $p' = \frac{p}{p-1}$. We 
always denote by $c$ and $C$ a positive constant, which is independent of main parameters but it 
may vary from line to line. However a constant with subscript or superscript
such as $c_1$ does not change in different occurrences. By $a \lesssim b$, ($b\gtrsim a$) we mean 
that $a\leq \lambda b$, where $\lambda > 0$ depends on inessential parameters. If $a \lesssim b$ 
and
$b\lesssim a$, we write $a\approx b$ and say that $a$ and $b$ are
equivalent.  We will denote by $\bf 1$ the function ${\bf 1}(x) =
1$, $x \in \mathbb R$. Since the expressions on our main results are too
long, to make the formulas plain we sometimes omit the differential
element $dx$. 

Now, we will present the main results of the paper. 

\begin{theorem}\label{maintheorem1}
	Let $0 < \theta \leq p < \min\{1, q\}$. Assume that $u, v \in \mathcal{M}^+$ and $w \in 
	\mathcal{W}$ such that $\int_t^{\infty} w < \infty$ for all $t \in (0, \infty)$.
	
	{\rm (i)} If $1 \leq q < \infty$, then inequality \eqref{main inequality} holds for all $f \in 
	\mathcal M^+$ if and only if 
	\begin{align*}
	A_1 :=  \sup_{x \in (0,\infty)} \bigg(\int_x^{\infty} w \bigg)^{-\frac{1}{\theta}} \sup_{t 
		\in (x,\infty)} \bigg(\int_x^t v^{\frac{1}{1-p}} \bigg)^{\frac{1-p}{p}} \bigg( 
	\int_t^{\infty} u\bigg)^{\frac{1}{q}} < \infty.
	\end{align*}
	Moreover, the best constant in \eqref{main inequality} satisfies $C \approx A_1$.
	
	{\rm (ii)} If $q < 1$, then inequality \eqref{main inequality} holds for all $f \in 
	\mathcal M^+$ if and only if 
	\begin{align*}
	A_2 :=  \sup_{x \in (0,\infty)} \bigg(\int_x^{\infty} w \bigg)^{-\frac{1}{\theta}} 
	\bigg(\int_x^{\infty} \bigg(\int_x^t v^{\frac{1}{1-p}} \bigg)^{\frac{q(1-p)}{p(1-q)}} 
	\bigg( \int_t^{\infty} u \bigg)^{\frac{q}{1-q}} u(t) dt \bigg)^{\frac{1-q}{q}} < \infty.
	\end{align*}
	Moreover, the best constant in \eqref{main inequality} satisfies $C \approx A_2$.
\end{theorem}

\begin{theorem}\label{maintheorem2}
	Let  $p=1$ and $0 < \theta \leq 1 < q < \infty$ . Assume that $u, v \in 
	\mathcal{M}^+$ and $w \in \mathcal{W}$  such that $\int_t^{\infty} w < \infty$ for 
	all $t \in (0, \infty)$. Then inequality \eqref{main inequality} holds for all $f \in 
	\mathcal M^+$ if and only if 
	\begin{align*}
	A_3 :=  \sup_{t \in (0,\infty)} \bigg( \int_t^{\infty} u\bigg)^{\frac{1}{q}} \esup_{s \in 
		(0, t)} v(s) \bigg(\int_s^{\infty} w \bigg)^{-\frac{1}{\theta}}   < \infty.
	\end{align*}
	Moreover, the best constant in \eqref{main inequality} satisfies $C \approx A_3$.
\end{theorem}

\begin{theorem}\label{maintheorem3}
	Let $0 < p < \min\{1,q, \theta\}$. Assume that $u, v \in \mathcal{M}^+$ and $w \in \mathcal{W}$ 
	such that $\int_t^{\infty} w < \infty$ for all $t \in (0, \infty)$. Suppose that
	\begin{equation*}
	0 < \bigg(\int_0^t \bigg(\int_s^t v^{\frac{1}{1-p}} \bigg)^{\frac{\theta(1-p)}{\theta - p}} 
	\bigg(\int_s^{\infty} w\bigg)^{-\frac{\theta}{\theta-p}} w(s) 
	ds\bigg)^{\frac{\theta-p}{\theta}} < \infty
	\end{equation*}
	holds for all $t \in (0,\infty)$.
	
	{\rm (i)} If $\max\{1, \theta\} \leq q < \infty$, then inequality \eqref{main inequality} 
	holds for all $f \in \mathcal M^+$ if and only if 
	\begin{align}\label{A_4}
	A_4 :=  \bigg(\int_0^{\infty} w \bigg)^{-\frac{1}{\theta}} \sup_{t \in (0,\infty)} 
	\bigg(\int_0^t v^{\frac{1}{1-p}} \bigg)^{\frac{1-p}{p}} \bigg( \int_t^{\infty} u 
	\bigg)^{\frac{1}{q}} < \infty,
	\end{align}
	and
	\begin{align}\label{A_{5}}
	A_{5} := \sup_{t \in (0,\infty)} \bigg( \int_0^t  \bigg(\int_s^{\infty} w 
	\bigg)^{-\frac{\theta}{\theta-p}} w(s) \bigg(\int_s^t v^{\frac{1}{1-p}} 
	\bigg)^{\frac{\theta(1-p)}{\theta-p}} ds \bigg)^{\frac{\theta-p}{\theta p}} 
	\bigg(\int_t^{\infty} u 
	\bigg)^{\frac{1}{q}}  < \infty.
	\end{align}
	Moreover, the best constant in \eqref{main inequality} satisfies $C \approx A_4 + A_{5}$.
	
	{\rm (ii)} If $1 \leq q < \theta < \infty$, then inequality \eqref{main inequality} holds for 
	all $f \in \mathcal M^+$ if and only if $A_4 < \infty$, 
	\begin{align*}
	A_{6} &:=  \left( \int_0^{\infty} \bigg( \int_0^t \bigg( \int_s^{\infty} w 
	\bigg)^{-\frac{\theta}{\theta - p}} w(s) ds \bigg)^{\frac{\theta(q-p)}{p(\theta - q)}}  
	\bigg( \int_t^{\infty} w \bigg)^{-\frac{\theta}{\theta - p}} w(t) \right. \notag\\
	& \hspace{2cm} \times \left. \sup_{z \in (t,\infty)} \bigg(\int_t^z v^{\frac{1}{1-p}} 
	\bigg)^{\frac{\theta q(1-p)}{p(\theta - q)}} 
	\bigg( \int_z^{\infty} u \bigg)^{\frac{\theta}{\theta - q}} dt 
	\right)^{\frac{\theta-q}{\theta q}}  < \infty,
	\end{align*}
	and
	\begin{align}\label{A_{7}}
	A_{7} &:=  \left( \int_0^{\infty} \bigg( \int_0^t \bigg( \int_s^{\infty} w 
	\bigg)^{-\frac{\theta}{\theta - p}} w(s) \bigg(\int_s^t v^{\frac{1}{1-p}} 
	\bigg)^{\frac{\theta(1-p)}{\theta - p}} ds \bigg)^{\frac{\theta(q-p)}{p(\theta - q)}}  
	\bigg( \int_t^{\infty} w \bigg)^{-\frac{\theta}{\theta - p}} w(t) \right. \notag\\
	& \hspace{2cm} \times \left. \sup_{z \in (t,\infty)} \bigg(\int_t^z v^{\frac{1}{1-p}} 
	\bigg)^{\frac{\theta(1-p)}{\theta - p}} \bigg( \int_z^{\infty} u 
	\bigg)^{\frac{\theta}{\theta - q}} dt \right)^{\frac{\theta-q}{\theta q}}  < \infty,
	\end{align}
	where $A_4$ is defined in \eqref{A_4}. Moreover, the best constant in \eqref{main inequality} 
	satisfies $C \approx A_4 + A_{6} + A_{7}$.
	
	{\rm (iii)} If $\theta \leq q < 1$, then inequality \eqref{main inequality} holds for 
	all $f \in \mathcal M^+$ if and only if $A_{5} < \infty$,
	\begin{align}\label{A_{8}}
	A_{8} :=  \bigg( \int_0^{\infty} w \bigg)^{-\frac{1}{\theta}}  \bigg( \int_0^{\infty} 
	\bigg(\int_0^t v^{\frac{1}{1-p}} \bigg)^{\frac{q(1-p)}{p(1-q)}} \bigg( \int_t^{\infty} u 
	\bigg)^{\frac{q}{1- q}} u(t) dt \bigg)^{\frac{1-q}{q}}  < \infty,
	\end{align}
	and
	\begin{align*}
	A_9 &:=  \sup_{t \in (0,\infty)} \bigg( \int_0^t \bigg( \int_s^{\infty} w \bigg)^{-\frac{\theta}{\theta - p}} w(s) ds \bigg)^{\frac{\theta-p}{\theta p}} \\
	&\hspace{2cm} \times \bigg( \int_t^{\infty} \bigg(\int_t^s v^{\frac{1}{1-p}} 
	\bigg)^{\frac{q(1-p)}{p(1-q)}} \bigg( \int_s^{\infty} u \bigg)^{\frac{q}{1-q}} u(s) ds 
	\bigg)^{\frac{1-q}{q}}  < \infty,
	\end{align*}
	where $A_{5}$ is defined in \eqref{A_{5}}. Moreover, the best constant in \eqref{main 
		inequality} satisfies $C \approx A_{5} + A_{8} + A_9$.
	
	{\rm (iv)} If $\theta < \infty$ and $ q < \min\{1,\theta\}$, then inequality \eqref{main 
		inequality} holds 
	for all$f \in \mathcal M^+$ if and only if  $A_{7} < \infty$, $A_{8} < \infty$ and
	\begin{align*}
	A_{10} &:=  \left( \int_0^{\infty} \bigg( \int_0^t \bigg( \int_s^{\infty} w 
	\bigg)^{-\frac{\theta}{\theta - p}} w(s) ds \bigg)^{\frac{\theta(q-p)}{p(\theta - q)}}  
	\bigg( \int_t^{\infty} w \bigg)^{-\frac{\theta}{\theta - p}} w(t) \right. \notag\\
	& \hspace{2cm} \times \left. \bigg( \int_t^{\infty} \bigg(\int_t^s v^{\frac{1}{1-p}} 
	\bigg)^{\frac{q(1-p)}{p(1-q)}} \bigg( \int_s^{\infty} u \bigg)^{\frac{q}{1 - q}} u(s) ds 
	\bigg)^{\frac{\theta(1-q)}{\theta-q}} dt\right)^{\frac{\theta-q}{\theta q}}  < \infty,
	\end{align*}
	where $A_{7}$ and $A_{8}$ are defined in \eqref{A_{7}} and \eqref{A_{8}}, respectively. 
	Moreover, the best constant in \eqref{main inequality} satisfies $C \approx A_{7} + A_{8} + 
	A_{10}$.
\end{theorem}

\begin{theorem}\label{maintheorem4}
	Let  $1 <  \min\{q, \theta\}$, $q, \theta < \infty$, and $p=1$. Assume that $u, v \in 
	\mathcal{M}^+$ and $w$ is a weight such that $\int_t^{\infty} w < \infty$ for all $t \in (0, 
	\infty)$. Suppose that $v$ is continuous and
	\begin{equation*}
	0 < \int_0^t v^{\frac{\theta}{\theta -1}} < \infty, \quad 0 < \int_0^t 
	\bigg(\int_x^{\infty} w \bigg)^{-\frac{\theta}{\theta-1}} w(x) dx < \infty, \quad 0 < 
	\int_0^t u^{-\frac{1}{q-1}} < \infty
	\end{equation*}
	holds for all $t \in (0,\infty)$. 
	
	{\rm (i)} If $\theta \leq q$, then inequality \eqref{main inequality} holds for all 
	$f \in \mathcal M^+$ if and only if 
	\begin{align}\label{A_11}
	A_{11} := \bigg(\int_0^{\infty} w \bigg)^{-\frac{1}{\theta}} \sup_{t \in (0,\infty)} \bigg( 
	\int_t^{\infty} u \bigg)^{\frac{1}{q}} \esup_{s \in (0, t)} v(s) < \infty,
	\end{align}
	and
	\begin{align*}
	A_{12} := \sup_{t \in (0,\infty)} \bigg( \int_0^t \bigg( \int_x^{\infty} w 
	\bigg)^{-\frac{\theta}{\theta -1}} w(x) \sup_{z \in (x,t)} v(z)^{\frac{\theta}{\theta-1}} dx 
	\bigg)^{\frac{\theta-1}{\theta}} \bigg( \int_t^{\infty} u\bigg)^{\frac{1}{q}}< \infty.
	\end{align*}
	Moreover, the best constant in \eqref{main inequality} satisfies $C \approx A_{11} + A_{12}$.
	
	{\rm (ii)} If $ q < \theta$, then inequality \eqref{main inequality} holds for all $f 
	\in \mathcal M^+$ if and only if $A_{11} < \infty$, 
	\begin{align*}
	A_{13} &:=  \left( \int_0^{\infty} \bigg( \int_0^t \bigg( \int_x^{\infty} w 
	\bigg)^{-\frac{\theta}{\theta - 1}} w(x) dx \bigg)^{\frac{\theta(q-1)}{\theta - q}}  \bigg( 
	\int_t^{\infty} w \bigg)^{-\frac{\theta}{\theta - 1}} w(t) \right. \\
	& \hspace{2cm} \times \left. \sup_{z \in (t,\infty)} v(z)^{\frac{\theta q}{\theta -q}} 
	\bigg( \int_z^{\infty} u \bigg)^{\frac{\theta}{\theta - q}}  dt \right)^{\frac{\theta-q}{\theta 
			q}}    < \infty,
	\end{align*}
	and
	\begin{align*}
	A_{14} &:=  \left( \int_0^{\infty} \bigg( \int_0^t \bigg( \int_x^{\infty} w 
	\bigg)^{-\frac{\theta}{\theta - 1}} w(x) \sup_{z \in (x,t)} v(z)^{\frac{\theta}{\theta-1}} dx 
	\bigg)^{\frac{\theta(q-1)}{\theta - q}}  \bigg( \int_t^{\infty} w \bigg)^{-\frac{\theta}{\theta 
			- 1}} w(t) \right. \\
	& \hspace{2cm} \times \left. \sup_{z \in (t,\infty)} v(z)^{\frac{\theta}{\theta-1}} \bigg( 
	\int_z^{\infty} u 
	\bigg)^{\frac{\theta}{\theta - q}}  dt \right)^{\frac{\theta-q}{\theta q}}  < \infty,
	\end{align*}
	where $A_{11}$ is defined in \eqref{A_11}. Moreover, the best constant in \eqref{main 
		inequality} satisfies $C \approx A_{11} + A_{13} + A_{14}$.
\end{theorem}

It should be noted that, using the characterization of the embedding between weighted Ces\`{a}ro 
function spaces, one can obtain the characterization of the embedding between weighted Copson 
function spaces. Indeed, using change of variables $x= 1/t$, it is easy to see that the 
embedding    
\begin{equation*}
\Cop_{p_1, q_1}(u_1, v_1) \hookrightarrow \Cop_{p_2, q_2}(u_2, v_2)
\end{equation*}
is equivalent to the embedding
\begin{equation*}
\Ces_{p_1, q_1}(\tilde u_1, \tilde v_1) \hookrightarrow \Cop_{p_2, q_2}(\tilde u_2, \tilde v_2),
\end{equation*}
where $\tilde u_i(t) = t^{-2/q_i} u_i(1/t)$ and $\tilde v_i(t) = t^{-2/p_i} v_i(1/t)$, $i= 1,2$, 
$t>0$. We will not formulate the results here. 

The paper is orginized as follows. In the second section we present necessary background materials. 
In the third section we prove the main results of this paper. 

\section{Definitions and Preliminaries}\label{sec:prelim}

Now, we will present some background information we need to prove our main results. Let us begin 
with the characterization of the well known Hardy-type inequalities (see, for instance, \cite{ok}, 
Section 1.)  

\begin{theorem}\label{Hardy Inequality q<i}
	Assume that $1 \leq p < \infty$, $0 < q < \infty$ and $v, w \in \mathcal{M}^+$. Let
	\begin{equation*}
	H = \sup_{f\in \mathcal M^+} \ddfrac{\bigg( \int_{0}^{\infty} \bigg( \int_t^{\infty} f(s) 
		ds \bigg)^q w(t) dt \bigg)^{\frac{1}{q}}}{\bigg( \int_0^{\infty} f(t)^p v(t) dt 
		\bigg)^{\frac{1}{p}}}.
	\end{equation*}
	
	{\rm (i)} If $1 < p \leq q $, then $H \approx H_1$, where
	\begin{equation*}
	H_1 = \sup_{t \in (0,\infty)} \bigg( \int_0^t w(s) ds \bigg)^{\frac{1}{q}} \bigg( 
	\int_t^{\infty} v(s)^{1-p'} ds \bigg)^{\frac{1}{p'}}.
	\end{equation*}
	
	{\rm (ii)} If $1 < p$ and $q < p$, then $H \approx H_2$, where
	\begin{equation*}
	H_2 =  \bigg( \int_0^{\infty} \bigg( \int_0^t w(s) ds \bigg)^{\frac{p}{p-q}} \bigg( 
	\int_t^{\infty} v(s)^{1-p'} ds \bigg)^{\frac{p(q-1)}{p-q}} v(t)^{1-p'} dt 
	\bigg)^{\frac{p-q}{pq}}.
	\end{equation*}
	
\end{theorem}

\begin{theorem}\label{Hardy Inequality q=i}
	Assume that $1 < p < \infty$ and $v, w \in \mathcal{M}^+$. Let
	\begin{equation*}
	H = \sup_{f\in \mathcal M^+} \ddfrac{\esup_{t\in (0,\infty)} \bigg( \int_t^{\infty} f(s) ds 
		\bigg) w(t)}{\bigg( \int_0^{\infty} f(t)^p v(t) dt \bigg)^{\frac{1}{p}}}.
	\end{equation*}
	Then $H \approx H_5$, where
	\begin{equation*}
	H_3 = \sup_{t \in (0,\infty)} \bigg( \esup_{s \in (0, t)} w(s) \bigg) \bigg( 
	\int_t^{\infty} v(s)^{1-p'} ds \bigg)^{\frac{1}{p'}}.
	\end{equation*}
\end{theorem}

Let us now recall the characterizations of reverse Hardy-type inequalities. 

\begin{theorem}\cite[Theorem 5.1]{ego2008} \label{Reverse Hardy Inequality q < p}
	Assume that $0 < q \leq p \leq 1$. Suppose that  $v, w \in \mathcal{M}^+$ such that $w$ 
	satisfies $\int_t^{\infty} w < \infty$ for all $t \in (0,\infty)$. Let 
	\begin{equation}\label{Rev.Hardy inequality R}
	R = \sup_{f\in \mathcal M^+} \ddfrac{\bigg( \int_0^{\infty} f(t)^p v(t) dt 
		\bigg)^{\frac{1}{p}}}{ \bigg( \int_{0}^{\infty} \bigg( \int_0^t f(s) ds \bigg)^q w(t) dt 
		\bigg)^{\frac{1}{q}}}.
	\end{equation}
	
	\rm{(i)} If $p < 1$, then $R \approx R_1$, where
	\begin{equation*}
	R_1 = \sup_{t \in (0,\infty)} \bigg(\int_t^{\infty} w(s) ds \bigg)^{-\frac{1}{q}} 
	\bigg(\int_t^{\infty} v(s)^{\frac{1}{1-p}} ds\bigg)^{\frac{1-p}{p}}.
	\end{equation*}
	
	\rm{(ii)} If $p=1$, then $R \approx R_2$, where
	\begin{equation*}
	R_2 = \sup_{t \in (0,\infty)} \bigg(\int_t^{\infty} w(s) ds \bigg)^{-\frac{1}{q}} 
	\bigg(\esup_{s \in (t,\infty)} v(s) \bigg).
	\end{equation*}
\end{theorem}

\begin{theorem}\cite[Theorem 5.4]{ego2008} \label{Reverse Hardy Inequality p < q}
	Assume that $0 < p \leq 1$ and $p < q < \infty$. Suppose that  $v, w \in \mathcal{M}^+$ such 
	that $w$ satisfies $\int_t^{\infty} w < \infty$ for all $t \in (0, \infty)$ and $w \neq 0$ a.e. 
	on $(0, \infty)$. Let $R$ be defined by \eqref{Rev.Hardy inequality R}.
	
	\rm{(i)} If $p < 1$, then $R \approx R_3$, where
	\begin{align*}
	R_3 & = \bigg( \int_0^{\infty} \bigg( \int_t^{\infty} v(s)^{\frac{1}{1-p}} ds 
	\bigg)^{\frac{q(1-p)}{q-p}} \bigg(\int_t^{\infty} w(s) ds\bigg)^{-\frac{q}{q-p}} w(t) 
	dt\bigg)^{\frac{q-p}{qp}} \\
	&\hspace{2cm} + \bigg( \int_0^{\infty} v(s)^{\frac{1}{1-p}} ds \bigg)^{\frac{1-p}{p}} 
	\bigg(\int_0^{\infty} w(s) ds\bigg)^{-\frac{1}{q}}.
	\end{align*}
	
	\rm{(ii)} If $p=1$, then $R \approx R_4$, where
	\begin{align*}
	R_4 & = \bigg( \int_0^{\infty} \bigg( \esup_{s \in (t,\infty)} v(s)^{\frac{q}{q-1}} \bigg) 
	\bigg(\int_t^{\infty} w(s) ds\bigg)^{-\frac{q}{q-1}} w(t) dt\bigg)^{\frac{q-1}{q}} \\
	&\hspace{2cm} + \bigg( \esup_{s \in (0,\infty)} v(s)\bigg) \bigg(\int_0^{\infty} w(s) 
	ds\bigg)^{-\frac{1}{q}}.
	\end{align*}
\end{theorem}

\begin{theorem}\cite[Theorem 1.1]{KrePick} \label{KrepelaPick Theorem}
	Let $1 < p < \infty$ and $0 < q, m < \infty $ and define $r:= \frac{pq}{p-q}$. Assume that  $u, 
	v, w \in \mathcal{M}^+$ such that 
	\begin{equation*}
	0 < \bigg(\int_0^t \bigg(\int_s^t u\bigg)^{\frac{q}{m}} w(s) ds\bigg)^{\frac{1}{q}} < \infty
	\end{equation*}
	for all $t \in (0, \infty)$. Let
	\begin{equation*}
	I = \sup_{f\in \mathcal M^+} \ddfrac{\bigg( \int_0^{\infty} \bigg(\int_t^{\infty} \bigg( 
		\int_s^{\infty} f\bigg)^m u(s) ds \bigg)^{\frac{q}{m}} w(t) dt\bigg)^{\frac{1}{q}}}
	{\bigg(\int_0^{\infty} f(t)^p v(t) dt\bigg)^{\frac{1}{p}}}.
	\end{equation*}
	
	{\rm (i)} If $p \leq \min\{m, q\}$, then $I \approx I_1$, where 
	\begin{equation}\label{I1}
	I_1 := \sup_{t \in (0,\infty)} \bigg( \int_0^t w(s) \bigg(\int_s^t u \bigg)^{\frac{q}{m}} 
	ds \bigg)^{\frac{1}{q}} \bigg(\int_t^{\infty} v^{1-p'} \bigg)^{\frac{1}{p'}}.
	\end{equation}
	
	{\rm (ii)} If $q < p \leq m$, then $ I \approx I_2 + I_3$, where
	\begin{equation*}
	I_2 := \bigg( \int_0^{\infty} \bigg( \int_0^t w\bigg)^{\frac{r}{p}} w(t) \sup_{z \in 
		(t,\infty)} \bigg(\int_t^z u \bigg)^{\frac{r}{m}} \bigg( \int_z^{\infty} v^{1-p'}  
	\bigg)^{\frac{r}{p'}} dt \bigg)^{\frac{1}{r}},
	\end{equation*}
	and 
	\begin{align}\label{I3}
	I_3 &:= \bigg( \int_0^{\infty} \sup_{z \in (t,\infty)} \bigg(\int_t^z u 
	\bigg)^{\frac{q}{m}} \bigg(\int_z^{\infty} v^{1-p'} \bigg)^{\frac{r}{p'}} \bigg.\notag \\
	&\hspace{2cm} \times\bigg.\bigg( \int_0^t w(s) \bigg(\int_s^t u\bigg)^{\frac{q}{m}} ds 
	\bigg)^{\frac{r}{p}} w(t) dt \bigg)^{\frac{1}{r}}.
	\end{align}
	
	{\rm (iii)}  If $ m < p \leq q$, then  $I \approx I_1 + I_4$, where $I_1$ is defined in 
	\eqref{I1} and
	\begin{align*}
	I_4  := \sup_{t \in (0,\infty)} \bigg( \int_0^t w \bigg)^{\frac{1}{q}} 
	\bigg(\int_t^{\infty} \bigg(\int_t^s u \bigg)^{\frac{p}{p-m}} \bigg(\int_s^{\infty} 
	v^{1-p'}  \bigg)^{\frac{p(m-1)}{p-m}} v(s)^{1-p'} ds \bigg)^{\frac{p-m}{pm}}.
	\end{align*}
	
	{\rm (iv)} If $\max\{m,q\} < p$ then $I \approx I_3 + I_5$, where $I_3$ is defined in 
	\eqref{I3} and
	\begin{align*}
	I_5 &:= \bigg( \int_{0}^{\infty} \bigg( \int_0^t w \bigg)^{\frac{r}{p}} w(t) \bigg.\\
	&\hspace{1cm}\times \bigg. \bigg(\int_t^{\infty} \bigg(\int_t^s u \bigg)^{\frac{p}{p-m}} 
	\bigg(\int_s^{\infty} v^{1-p'}  \bigg)^{\frac{p(m-1)}{p-m}} v(s)^{1-p'} ds 
	\bigg)^{\frac{q(p-m)}{m(p-q)}} dt \bigg)^{\frac{1}{r}}.
	\end{align*}
\end{theorem}

\begin{theorem}\cite[Theorem  6]{Krepela} \label{Krepela Theorem}
	Let $1 < p < \infty$ and $0 < q < \infty$ and set $r:= \frac{pq}{p-q}$. Assume that $u, v, w 
	\in \mathcal{M}^+$ such that $u$ is continuous and
	$$
	0 < \int_0^t u < \infty, \quad 0 < \int_0^t v < \infty, \quad 0 <\int_0^t w < \infty 
	$$
	hold for all $t \in (0, \infty)$. Let 
	\begin{equation*}
	I = \sup_{f\in \mathcal M^+} \ddfrac{\bigg( \int_0^{\infty} \bigg(\sup_{s \in (t,\infty)} 
		u(s) \int_s^{\infty} f \bigg)^q w(t) dt\bigg)^{\frac{1}{q}}}
	{\bigg(\int_0^{\infty} f(t)^p v(t) dt\bigg)^{\frac{1}{p}}}.
	\end{equation*}
	
	{\rm (i)}  If $p \leq q$ then $I \approx I_6$, where
	\begin{equation*}
	I_6 := \sup_{t \in (0,\infty)} \bigg(\int_0^t w(s) \sup_{z \in (s,t)} u(z)^ q ds 
	\bigg)^{\frac{1}{q}} \bigg( \int_t^{\infty} v^{1-p'} \bigg)^{\frac{1}{p'}}.
	\end{equation*}
	
	{\rm (ii)}  If $q < p$, then $I \approx I_7 + I_8$, where 
	\begin{equation*}
	I_7 := \bigg( \int_0^{\infty} \bigg(\int_0^t w \bigg)^{\frac{r}{p}} w(t) \sup_{s \in 
		(t,\infty)} u(s)^r \bigg( \int_s^{\infty} v^{1-p'} \bigg)^{\frac{r}{p'}} dt 
	\bigg)^{\frac{1}{r}},
	\end{equation*}
	and 
	\begin{equation*}
	I_8 := \bigg( \int_0^{\infty} \bigg(\int_0^t w(s) \sup_{z \in (s,t)} u(z)^q ds 
	\bigg)^{\frac{r}{p}} w(t)  \sup_{z \in (t,\infty)} u(z)^q \bigg( \int_z^{\infty} v^{1-p'} 
	ds \bigg)^{\frac{r}{p'}} dt \bigg)^{\frac{1}{r}}.
	\end{equation*}
\end{theorem}

\section{Proofs of the Main Results}\label{S:proofs}

In this section we will prove our main results. Following lemma explains the assumption $p<1$, when 
characterizing our main inequality.

\begin{lemma}\label{triviality}
	Let $0 < p, q, \theta < \infty$. Assume that $u, v, w \in \mathcal{M}^+$ such that 
	$$
	0 < \int_t^{\infty} u < \infty, \quad 0 < \int_t^{\infty} w < \infty, 
	$$
	for all $t \in (0, \infty)$. If $p >1$, then inequality \eqref{main inequality} holds only for 
	trivial functions.  
\end{lemma}

\begin{proof}
	Suppose that there exists a positive constant $C$ such that inequality \eqref{main inequality} 
	holds for all non-negative measurable functions on $(0,\infty)$. 
	
	Let $0 < \tau_1 < \tau_2 < \infty$ and assume that $h$ is a non-negative measurable function 
	such that $\supp h \subset [\tau_1, \tau_2]$. Testing inequality \eqref{main inequality} with 
	$h$, one can see that
	\begin{align*}
	\bigg(\int_0^{\infty} \bigg(\int_0^t h^p v\bigg)^{\frac{q}{p}} u(t) dt \bigg)^{\frac{1}{q}} 
	&\geq \bigg(\int_{\tau_2}^{\infty} \bigg(\int_0^t h^p v\bigg)^{\frac{q}{p}} u(t) dt 
	\bigg)^{\frac{1}{q}}\\
	& = \bigg(\int_{\tau_1}^{\tau_2} h^p v\bigg)^{\frac{1}{p}} \bigg(\int_{\tau_2}^{\infty} 
	u(t) dt \bigg)^{\frac{1}{q}}
	\end{align*}
	and
	\begin{align*}
	\bigg( \int_0^{\infty} \bigg(\int_0^t f(s)  ds \bigg)^{\theta} w(t) dt 
	\bigg)^{\frac{1}{\theta}} & = \bigg( \int_{\tau_1}^{\infty} \bigg(\int_0^t f(s)  ds 
	\bigg)^{\theta} w(t) dt \bigg)^{\frac{1}{\theta}} \\
	&\leq \bigg(\int_{\tau_1}^{\tau_2} h \bigg) \bigg(\int_{\tau_1}^{\infty} w 
	\bigg)^{\frac{1}{\theta}}
	\end{align*}
	hold. Hence the validity of inequality \eqref{main inequality} implies that
	\begin{equation*}
	\bigg(\int_{\tau_1}^{\tau_2} h^p v\bigg)^{\frac{q}{p}} \bigg(\int_{\tau_2}^{\infty} u(t) dt 
	\bigg)^{\frac{1}{q}} \leq C \bigg(\int_{\tau_1}^{\tau_2} h \bigg) 
	\bigg(\int_{\tau_1}^{\infty} w \bigg)^{\frac{1}{\theta}}.
	\end{equation*}
	Since $0 < \int_t^{\infty} u, \int_t^{\infty} w < \infty$ for all $t \in (0, \infty)$, we 
	arrive at $L_1({\bf 1}) \hookrightarrow L_{p}(v)$ when $p >1$, which is a contradiction. 
\end{proof}

\

{\bf \textit{Proof of Theorem~\ref{maintheorem1}}. }
We begin with the well-known duality principle in weighted Lebesgue spaces. Recall that $p \in 
(1,\infty)$, $f \in \mathcal{M}^+$ and $v$ is a weight on $(0,\infty)$, then
\begin{equation*}
\bigg(\int_0^{\infty} f(t)^p v(t) dt \bigg)^{\frac{1}{p}} = \sup_{h\in \mathcal M^+} 
\frac{\int_0^{\infty} f(t) h(t) dt}{\big(\int_0^{\infty} h(t)^{p'} v(t)^{1-p'} 
	dt\big)^{\frac{1}{p'}}}.
\end{equation*} 
It is clear that since in our case  $q/p >1$, using Sawyer duality, the best constant of inequality 
\eqref{main inequality} satisfies
\begin{align*}
C = \sup_{f \in \mathcal M^+} \ddfrac{1}{\bigg(\int_0^{\infty} \bigg(\int_0^t f \bigg)^{\theta} 
	w(t) dt\bigg)^{\frac{1}{\theta}}}  \sup_{h\in \mathcal  M^+} \ddfrac{\bigg(\int_0^{\infty} h(t) 
	\int_0^t f(s)^p v(s) ds \, dt\bigg)^{\frac{1}{p}}}{\bigg( \int_0^{\infty} h(t)^{\frac{q}{q-p}} 
	u(t)^{-\frac{p}{q-p}} dt\bigg)^{\frac{q-p}{qp}}}.
\end{align*}
Interchanging suprema and applying Fubini, we get that
\begin{align*}
C = \sup_{h\in \mathcal  M^+} \ddfrac{1}{\bigg( \int_0^{\infty} h(t)^{\frac{q}{q-p}} 
	u(t)^{-\frac{p}{q-p}} dt\bigg)^{\frac{q-p}{qp}}} \sup_{f \in \mathcal M^+} 
\ddfrac{\bigg(\int_0^{\infty} f(t)^p v(t) \int_t^{\infty} h(s) ds 
	dt\bigg)^{\frac{1}{p}}}{\bigg(\int_0^{\infty} \bigg(\int_0^t f \bigg)^{\theta} w(t) 
	dt\bigg)^{\frac{1}{\theta}}}.
\end{align*}
Denote by
\begin{equation}\label{D}
D := \sup_{f \in \mathcal M^+} \ddfrac{\bigg(\int_0^{\infty} f(t)^p v(t) \int_t^{\infty} h(s) 
	ds dt\bigg)^{\frac{1}{p}}}{\bigg(\int_0^{\infty} \bigg(\int_0^t f \bigg)^{\theta} w(t) 
	dt\bigg)^{\frac{1}{\theta}}}.
\end{equation}
Since $\theta \leq p < 1$, we have by applying [Theorem~\ref{Reverse Hardy Inequality q < p}, (i)] 
that
\begin{equation*}
D \approx \sup_{x \in (0,\infty)} \bigg( \int_x^{\infty} v(s)^{\frac{1}{1-p}} 
\bigg(\int_s^{\infty} h \bigg)^{\frac{1}{1-p}} ds \bigg)^{\frac{1-p}{p}} \bigg(\int_x^{\infty} 
w \bigg)^{-\frac{1}{\theta}}.
\end{equation*}
Therefore, 
\begin{equation*}
C \approx \sup_{h\in \mathcal  M^+} \ddfrac{ \sup_{x \in (0,\infty)} \bigg( \int_x^{\infty} 
	v(s)^{\frac{1}{1-p}} \bigg(\int_s^{\infty} h \bigg)^{\frac{1}{1-p}} ds \bigg)^{\frac{1-p}{p}} 
	\bigg(\int_x^{\infty} w \bigg)^{-\frac{1}{\theta}}}{\bigg( \int_0^{\infty} h(t)^{\frac{q}{q-p}} 
	u(t)^{-\frac{p}{q-p}} dt\bigg)^{\frac{q-p}{qp}}}.
\end{equation*}
Interchanging suprema yields that
\begin{equation*}
C \approx \sup_{x \in (0,\infty)} \bigg(\int_x^{\infty} w \bigg)^{-\frac{1}{\theta}} \sup_{h\in 
	\mathcal  M^+} \ddfrac{  \bigg( \int_0^{\infty}\bigg(\int_s^{\infty} h \bigg)^{\frac{1}{1-p}}  
	v(s)^{\frac{1}{1-p}} \chi_{(x,\infty)}(s) \, ds \bigg)^{\frac{1-p}{p}} }{\bigg( \int_0^{\infty} 
	h(t)^{\frac{q}{q-p}} u(t)^{-\frac{p}{q-p}} dt\bigg)^{\frac{q-p}{qp}}}.
\end{equation*}
Then, it remains to apply Theorem~\ref{Hardy Inequality q<i}. To this end, we need to split into 
two cases.

{\rm(i)} If $1 \leq q$, in this case $\frac{1}{1-p} \geq \frac{q}{q-p}$, then applying  
[Theorem~\ref{Hardy Inequality q<i}, (i)], we obtain that
\begin{align*}
C & \approx  \sup_{x \in (0,\infty)} \bigg(\int_x^{\infty} w \bigg)^{-\frac{1}{\theta}} \sup_{t 
	\in (0,\infty)} \bigg(\int_0^t v(s)^{\frac{1}{1-p}} \chi_{(x,\infty)}(s) ds 
\bigg)^{\frac{1-p}{p}} \bigg( \int_t^{\infty} u\bigg)^{\frac{1}{q}}\\
& = \sup_{x \in (0,\infty)} \bigg(\int_x^{\infty} w \bigg)^{-\frac{1}{\theta}} \max \left\{ 
\sup_{t \in (0,x)} \bigg(\int_0^t v(s)^{\frac{1}{1-p}} \chi_{(x,\infty)}(s) ds 
\bigg)^{\frac{1-p}{p}} \bigg( \int_t^{\infty} u\bigg)^{\frac{1}{q}},  \right. \\
& \hspace{4cm}  \left. \sup_{t \in (x,\infty)} \bigg(\int_0^t v(s)^{\frac{1}{1-p}} 
\chi_{(x,\infty)}(s) ds \bigg)^{\frac{1-p}{p}} \bigg( \int_t^{\infty} 
u\bigg)^{\frac{1}{q}}\right\}\\
& = \sup_{x \in (0,\infty)} \bigg(\int_x^{\infty} w \bigg)^{-\frac{1}{\theta}} \sup_{t \in 
	(x,\infty)} \bigg(\int_x^t  v^{\frac{1}{1-p}} \bigg)^{\frac{1-p}{p}} \bigg( \int_t^{\infty} 
u\bigg)^{\frac{1}{q}}
\end{align*}

{\rm(ii)} If $q < 1$, in this case $\frac{1}{1-p} < \frac{q}{q-p}$, then applying  
[Theorem~\ref{Hardy Inequality q<i}, (ii)], we arrive at 
\begin{align*}
C & \approx   \sup_{x \in (0,\infty)} \bigg(\int_x^{\infty} w \bigg)^{-\frac{1}{\theta}} 
\bigg(\int_0^{\infty} \bigg(\int_0^t v^{\frac{1}{1-p}} \chi_{(x,\infty)}(s) ds 
\bigg)^{\frac{q(1-p)}{p(1-q)}} \bigg( \int_t^{\infty} u \bigg)^{\frac{q}{1-q}} u(t) dt 
\bigg)^{\frac{1-q}{q}} \\
& = \sup_{x \in (0,\infty)} \bigg(\int_x^{\infty} w \bigg)^{-\frac{1}{\theta}} 
\bigg(\int_x^{\infty} \bigg(\int_x^t v^{\frac{1}{1-p}}
\bigg)^{\frac{q(1-p)}{p(1-q)}} \bigg( \int_t^{\infty} u \bigg)^{\frac{q}{1-q}} u(t) dt 
\bigg)^{\frac{1-q}{q}}.
\end{align*}

\

{\bf \textit{Proof of Theorem~\ref{maintheorem2}}. }
As in the previous proof since $q / p > 1 $, duality approach yields that,
\begin{equation*}
C =  \sup_{h\in \mathcal  M^+} \ddfrac{D}{\bigg( \int_0^{\infty} h(t)^{\frac{q}{q-p}} 
	u(t)^{-\frac{p}{q-p}} dt\bigg)^{\frac{q-p}{qp}}},
\end{equation*}
where $D$ is defined in \eqref{D}. Since, in this case $\theta \leq p =1$, we have by applying 
[Theorem~\ref{Reverse Hardy Inequality q < p}, (ii)], that
\begin{align*}
C \approx  \sup_{h\in \mathcal  M^+} \ddfrac{ \sup_{x \in (0,\infty)} \bigg(\int_x^{\infty} 
	w\bigg)^{-\frac{1}{\theta}} \esup_{s \in (x,\infty)} v(s) \int_s^{\infty} h}{\bigg( 
	\int_0^{\infty} h(t)^{\frac{q}{q-p}} u(t)^{-\frac{p}{q-p}} dt\bigg)^{\frac{q-p}{qp}}}.
\end{align*} 

Recall that if $F$ is a non-negative, non-decreasing measurable function on $(0,\infty)$, then 
\begin{equation}\label{Fubini sup inc.}
\esup_{t\in (0,\infty)} F(t) G(t) = \esup_{t\in (0,\infty)} F(t) \esup_{\tau \in (t,\infty)} 
G(\tau), 
\end{equation}
holds (see, for instance, page 85 in \cite{GogPick}). On using \eqref{Fubini sup inc.}, we obtain 
that
\begin{align*}
C \approx  \sup_{h\in \mathcal  M^+} \ddfrac{ \sup_{x \in (0,\infty)} \bigg(\int_x^{\infty} 
	w\bigg)^{-\frac{1}{\theta}}  v(x) \int_x^{\infty} h} {\bigg( \int_0^{\infty} 
	h(t)^{\frac{q}{q-p}} u(t)^{-\frac{p}{q-p}} dt\bigg)^{\frac{q-p}{qp}}}.
\end{align*} 
Finally, applying Theorem~\ref{Hardy Inequality q=i}, we arrive at
\begin{equation*}
C \approx \sup_{t \in (0,\infty)} \bigg( \int_t^{\infty} u\bigg)^{\frac{1}{q}} \esup_{s \in (0, 
	t)} v(s) \bigg(\int_s^{\infty} w \bigg)^{-\frac{1}{\theta}}.
\end{equation*}

\

{\bf \textit{Proof of Theorem~\ref{maintheorem3}}. }
Similar to the previous proofs, we have that
\begin{equation*}
C =  \sup_{h\in \mathcal  M^+} \ddfrac{D}{\bigg( \int_0^{\infty} h(t)^{\frac{q}{q-p}} 
	u(t)^{-\frac{p}{q-p}} dt\bigg)^{\frac{q-p}{qp}}},
\end{equation*}
where $D$ is defined in \eqref{D}. Since in this case $p < 1$ and $p < \theta$, we have, by 
applying [Theorem~\ref{Reverse Hardy Inequality p < q}, (i)], that
\begin{align*}
C &\approx  \bigg(\int_0^{\infty} w \bigg)^{-\frac{1}{\theta}}\sup_{h\in \mathcal  M^+} 
\ddfrac{\bigg( \int_0^{\infty} \bigg(\int_s^{\infty} h \bigg)^{\frac{1}{1-p}} 
	v(s)^{\frac{1}{1-p}} ds \bigg)^{\frac{1-p}{p}}}{\bigg( \int_0^{\infty} h(t)^{\frac{q}{q-p}} 
	u(t)^{-\frac{p}{q-p}} dt\bigg)^{\frac{q-p}{qp}}} \\ 
&\hspace{1cm}+  \sup_{h\in \mathcal  M^+} \ddfrac{\bigg(\int_0^{\infty} \bigg( \int_x^{\infty} 
	\bigg(\int_t^{\infty} h \bigg)^{\frac{1}{1-p}} v(t)^{\frac{1}{1-p}} dt \bigg)^{\frac{\theta 
			(1-p)}{\theta -p}}  \bigg( \int_x^{\infty} w\bigg)^{-\frac{\theta}{\theta-p}} w(x) dx 
	\bigg)^{\frac{\theta - p}{\theta p}}}{\bigg( \int_0^{\infty} h(t)^{\frac{q}{q-p}} 
	u(t)^{-\frac{p}{q-p}} dt\bigg)^{\frac{q-p}{qp}}}\\
&=: C_1 + C_2.
\end{align*}

Let us first consider $C_1$. We need to consider the cases $q < 1$ and $1 \leq q $ seperately. 
Hence, we begin with the condition $p < 1 \leq q$. Using [Theorem~\ref{Hardy Inequality q<i}, (i)], 
we obtain that
\begin{equation}\label{C1 1<q}
C_1 \approx \bigg(\int_0^{\infty} w\bigg)^{-\frac{1}{\theta}} \sup_{t \in (0,\infty)} 
\bigg(\int_0^t v^{\frac{1}{1-p}}\bigg)^{\frac{1-p}{p}} \bigg(\int_t^{\infty} 
u\bigg)^{\frac{1}{q}}=: A_4. 
\end{equation}
On the other hand, if $p < q < 1$, using [Theorem~\ref{Hardy Inequality q<i}, (ii)], we get that
\begin{equation}\label{C1 q<1}
C_1 \approx \bigg(\int_0^{\infty} w\bigg)^{-\frac{1}{\theta}} \bigg(\int_0^{\infty} 
\bigg(\int_0^t 
v^{\frac{1}{1-p}}\bigg)^{\frac{q(1-p)}{p(1-q)}} \bigg(\int_t^{\infty} u\bigg)^{\frac{q}{1-q}} 
u(t) 
dt\bigg)^{\frac{1-q}{q}}=: A_{8}. 
\end{equation}

Let us now eveluate $C_2$. We will apply Theorem~\ref{KrepelaPick Theorem} with parameters 
$$
m= \frac{1}{1-p}, \quad q=\frac{\theta}{\theta - p}, \quad p=\frac{q}{q-p}.
$$
Thus, we need to consider the conditions on parameters in four cases.

{\rm (i)} If $p < \min\{1,q, \theta\}$ and $\max\{1, \theta\} \leq q$, then applying 
[Theorem~\ref{KrepelaPick Theorem}, (i)], we get that $C_2 \approx I_1^{\frac{1}{p}}$, where
\begin{equation}\label{I1 proof}
I_1^{\frac{1}{p}} = \sup_{t \in (0,\infty)} \bigg( \int_0^t  \bigg(\int_s^{\infty} w 
\bigg)^{-\frac{\theta}{\theta-p}} w(s) \bigg(\int_s^t v^{\frac{1}{1-p}} 
\bigg)^{\frac{\theta(1-p)}{\theta-p}} ds \bigg)^{\frac{\theta-p}{\theta p}} 
\bigg(\int_t^{\infty} u 
\bigg)^{\frac{1}{q}} = A_{5}.
\end{equation}
Then, since $1 < q$ in this case, we have that $C_1 \approx A_4$. Therefore  $C = C_1 + C_2 \approx 
A_4 + A_{5}$. 

{\rm (ii)} If $p < \min\{1,q, \theta\}$ and $1 \leq q < \theta$, then applying 
[Theorem~\ref{KrepelaPick Theorem}, (ii)], we get that $C_2 \approx I_2^{\frac{1}{p}} + 
I_3^{\frac{1}{p}}$, where
\begin{align*}
I_2^{\frac{1}{p}} &=  \left( \int_0^{\infty} \bigg( \int_0^t \bigg( \int_s^{\infty} w 
\bigg)^{-\frac{\theta}{\theta - p}} w(s) ds \bigg)^{\frac{\theta(q-p)}{p(\theta - q)}}  \bigg( 
\int_t^{\infty} w \bigg)^{-\frac{\theta}{\theta - p}} w(t) \right. \notag\\
& \hspace{2cm} \times \left. \sup_{z \in (t,\infty)} \bigg(\int_t^z v^{\frac{1}{1-p}} 
\bigg)^{\frac{\theta q(1-p)}{p(\theta - q)}} 
\bigg( \int_z^{\infty} u \bigg)^{\frac{\theta}{\theta - q}} dt \right)^{\frac{\theta-q}{\theta 
		q}} 
=: A_{6}.
\end{align*}
and
\begin{align}\label{I3 proof}
I_3^{\frac{1}{p}} &=  \left( \int_0^{\infty} \bigg( \int_0^t \bigg( \int_s^{\infty} w 
\bigg)^{-\frac{\theta}{\theta - p}} w(s) \bigg(\int_s^t v^{\frac{1}{1-p}} 
\bigg)^{\frac{\theta(q-p)}{\theta - p}} ds \bigg)^{\frac{\theta(q-p)}{p(\theta - q)}}  \bigg( 
\int_t^{\infty} w \bigg)^{-\frac{\theta}{\theta - p}} w(t) \right. \notag\\
& \hspace{2cm} \times \left. \sup_{z \in (t,\infty)} \bigg(\int_t^z v^{\frac{1}{1-p}} 
\bigg)^{\frac{\theta(1-p)}{\theta - p}} \bigg( \int_z^{\infty} u \bigg)^{\frac{\theta}{\theta - 
		q}} 
dt \right)^{\frac{\theta-q}{\theta q}} =: A_{7}.
\end{align}
Since, $1 < q$ in this case, we have that $C_1 \approx A_4$. Therefore  $C = C_1 + C_2 \approx A_4 
+ A_{6} + A_{7}$. 

{\rm (iii)} If $p < \min\{1,q, \theta\}$ and $\theta \leq q < 1$, then applying 
[Theorem~\ref{KrepelaPick Theorem}, (iii)], we get that $C_2 \approx I_1^{\frac{1}{p}} + 
I_4^{\frac{1}{p}}$, where $I_1$ is given in \eqref{I1 proof} and
\begin{align*}
I_4^{\frac{1}{p}} &:=  \sup_{t \in (0,\infty)} \bigg( \int_0^t \bigg( \int_s^{\infty} w 
\bigg)^{-\frac{\theta}{\theta - p}} w(s) ds \bigg)^{\frac{\theta-p}{\theta p}} \\
&\hspace{2cm} \times \bigg( \int_t^{\infty} \bigg(\int_t^s v^{\frac{1}{1-p}} 
\bigg)^{\frac{q(1-p)}{p(1-q)}} \bigg( \int_s^{\infty} u \bigg)^{\frac{q}{1-q}} u(s) ds 
\bigg)^{\frac{1-q}{q}} \\
&=: A_9.
\end{align*}
Since $q < 1$, we have that $C_1 \approx A_{8}$. Thus, $C = C_1 + C_2 \approx A_{8} + A_{5} + A_9$.

{\rm (iv)} If $p < \min\{1,q, \theta\}$ and $q < \min\{1,\theta\}$, then applying 
[Theorem~\ref{KrepelaPick Theorem}, (iv)], we get that $C_2 \approx I_3^{\frac{1}{p}} + 
I_5^{\frac{1}{p}}$, where $I_3$ is given in \eqref{I3 proof} and
\begin{align*}
I_5^{\frac{1}{p}} &:=  \left( \int_0^{\infty} \bigg( \int_0^t \bigg( \int_s^{\infty} w 
\bigg)^{-\frac{\theta}{\theta - p}} w(s) ds \bigg)^{\frac{\theta(q-p)}{p(\theta - q)}}  \bigg( 
\int_t^{\infty} w \bigg)^{-\frac{\theta}{\theta - p}} w(t) \right. \notag\\
& \hspace{2cm} \times \left. \bigg( \int_t^{\infty} \bigg(\int_t^s v^{\frac{1}{1-p}} 
\bigg)^{\frac{q(1-p)}{p(1-q)}} \bigg( \int_s^{\infty} u \bigg)^{\frac{q}{1 - q}} u(s) ds 
\bigg)^{\frac{\theta(1-q)}{\theta-q}} dt\right)^{\frac{\theta-q}{\theta q}} =: A_{10}.
\end{align*}
Since, $q < 1$, again we have that $C_1 \approx A_{8}$, which yields that $C = C_1 + C_2 \approx 
A_{8} + A_{7} + A_{10}$.


{\bf \textit{Proof of Theorem~\ref{maintheorem4}}. } 
We have already shown that
\begin{equation*}
C =  \sup_{h\in \mathcal  M^+} \ddfrac{D}{\bigg( \int_0^{\infty} h(t)^{\frac{q}{q-p}} 
	u(t)^{-\frac{p}{q-p}} dt\bigg)^{\frac{q-p}{qp}}},
\end{equation*}
where $D$ is defined in \eqref{D}. Since in this case $p = 1$ and $p < \theta$, we have, by 
applying [Theorem~\ref{Reverse Hardy Inequality p < q}, (ii)], that
\begin{align*}
C &\approx  \bigg(\int_0^{\infty} w \bigg)^{-\frac{1}{\theta}} \sup_{h\in \mathcal  M^+} 
\ddfrac{\esup_{x \in (0,\infty)} v(x) \int_x^{\infty} h }{\bigg( \int_0^{\infty} 
	h(t)^{\frac{q}{q-p}} u(t)^{-\frac{p}{q-p}} dt\bigg)^{\frac{q-p}{qp}}} \\ 
&\hspace{1cm}+  \sup_{h\in \mathcal  M^+} \ddfrac{\bigg( \int_0^{\infty} \bigg( \esup_{s \in 
		(x,\infty)} v(s)\int_s^{\infty} h   \bigg)^{\frac{\theta}{\theta-1}}  
	\bigg( \int_x^{\infty} w \bigg)^{-\frac{\theta}{\theta-1}} w(x) dx \bigg)^{\frac{\theta - 
			1}{\theta }}}{\bigg( \int_0^{\infty} h(t)^{\frac{q}{q-1}} u(t)^{-\frac{1}{q-1}} 
	dt\bigg)^{\frac{q-1}{q}}}\\
&=: C_3 + C_4.
\end{align*}
Since $q/q-1 > 1$, applying Theorem~\ref{Hardy Inequality q=i}, we have that
\begin{equation*}
C_3 \approx \bigg(\int_0^{\infty} w \bigg)^{-\frac{1}{\theta}} \sup_{t \in (0,\infty)} 
\bigg(\esup_{s \in (0, t)} v(s)\bigg) \bigg(\int_t^{\infty} u\bigg)^{\frac{1}{q}} =: A_{11}.
\end{equation*}
On the other hand, in order to calculate $C_4$, we will apply Theorem~\ref{Krepela Theorem} with 
parameters
$$
q= \frac{\theta}{\theta-1} \quad \text{and} \quad p=\frac{q}{q-1}.
$$
We need to apply this theorem to the cases $\theta \leq q$ and $q < \theta$ seperately. 

{\rm (i)} If $\theta \leq q$, then applying [Theorem~\ref{Krepela Theorem}, (i)], we have that $C_4 
\approx I_6$, where 
\begin{align*}
I_6 = \sup_{t \in (0,\infty)} \bigg( \int_0^t \bigg( \int_x^{\infty} w 
\bigg)^{-\frac{\theta}{\theta -1}} w(x) \sup_{z \in (x,t)} v(z)^{\frac{\theta}{\theta-1}} dx 
\bigg)^{\frac{\theta-1}{\theta}} \bigg( \int_t^{\infty} u\bigg)^{\frac{1}{q}} =: A_{12}.
\end{align*}
Therefore, $C = C_3 + C_4 \approx A_{11} + A_{12}$.

{\rm (ii)} If $q < \theta $, then applying [Theorem~\ref{Krepela Theorem}, (ii)], we have that $C_4 
\approx I_7 + I_8$, where
\begin{align*}
I_7 &=  \left( \int_0^{\infty} \bigg( \int_0^t \bigg( \int_x^{\infty} w 
\bigg)^{-\frac{\theta}{\theta - 1}} w(x) dx \bigg)^{\frac{\theta(q-1)}{\theta - q}}  \bigg( 
\int_t^{\infty} w \bigg)^{-\frac{\theta}{\theta - 1}} w(t) \right. \\
& \hspace{2cm} \times \left. \sup_{z \in (t,\infty)} v(z)^{\frac{\theta q}{\theta -q}} 
\bigg( \int_z^{\infty} u \bigg)^{\frac{\theta}{\theta - q}}  dt \right)^{\frac{\theta-q}{\theta 
		q}}  =: A_{13},
\end{align*}
and
\begin{align*}
I_8 &= \left( \int_0^{\infty} \bigg( \int_0^t \bigg( \int_x^{\infty} w 
\bigg)^{-\frac{\theta}{\theta - 1}} w(x) \sup_{z \in (x,t)} v(z)^{\frac{\theta}{\theta-1}} dx 
\bigg)^{\frac{\theta(q-1)}{\theta - q}}  \bigg( \int_t^{\infty} w \bigg)^{-\frac{\theta}{\theta 
		- 1}} w(t) \right. \\
& \hspace{2cm} \times \left. \sup_{z \in (t,\infty)} v(z)^{\frac{\theta}{\theta-1}} \bigg( 
\int_z^{\infty} u 
\bigg)^{\frac{\theta}{\theta - q}}  dt \right)^{\frac{\theta-q}{\theta q}}  =: A_{14}.
\end{align*}
Then, we arrive at $C = C_3 + C_4 \approx A_{11} + A_{13} + A_{14}$.

\

{\bf \textit{Acknowledgements}.} The author would like to thank Professor Amiran Gogatishvili for 
his 
valuable suggestions and guidance during the preperation of this manuscript.


\end{document}